\documentclass[12pt]{amsart}
\usepackage{amsfonts,amsmath,amssymb,amscd,mathrsfs}
\usepackage{xcolor}
\usepackage{verbatim}
\usepackage{enumitem}
\usepackage{hyperref}
\usepackage{tikz-cd}
\newtheorem{theorem}{Theorem}[section]
\newtheorem{lemma}[theorem]{Lemma}
\newtheorem{corollary}[theorem]{Corollary}
\newtheorem{proposition}[theorem]{Proposition}

\theoremstyle{definition}
\newtheorem{definition}[theorem]{Definition}

\newtheorem{remark}[theorem]{Remark}

\theoremstyle{remark}

\def \D{\mathbb{D}}

\def \C{\mathbb{C}}
\def \T{\mathbb{T}}

\newcommand{\clh}{\mathcal{H}}
\newcommand{\clm}{\mathcal{M}}
\newcommand{\cln}{\mathcal{N}}

\newcommand{\clk}{\mathcal{K}}

\newcommand{\ran}{\mathrm{ran}}

\newcommand{\vp}{\varphi}

\newcommand{\ol}{\overline}
\newcommand\restr[2]{\ensuremath{\left.#1\right|_{#2}}}

\numberwithin{equation}{section}


\date{\today}
\begin{document}
\title[Weighted DTTOs on nearly invariant subspaces]
{Nearly invariant subspaces and weighted dual truncated Toeplitz operators}
\author{Sudip Ranjan Bhuia}
\address{Department of Mathematics, Shiv Nadar University, NH91, Tehsil Dadri, Greater Noida, Gautam Buddha Nagar, Uttar Pradesh-201314, India}
\email{sudipranjanb@gmail.com; sudip.bhuia@snu.edu.in}

\begin{abstract}
Let $\clm=h\clk_u$ be a nearly $S^*$-invariant subspace of $H^2$, where $\clk_u=H^2\ominus uH^2$ and $h$ is the extremal multiplier. For $\vp\in L^\infty(\T)$, we study the compression
\[
D_\vp^\clm = \restr{P_{\clm^\perp}M_\vp}{\clm^\perp},
\]
called a weighted dual truncated Toeplitz operator. When $h\equiv1$, this reduces to the classical dual truncated Toeplitz operator.

Using the Hartmann--Ross projection formula, we prove
\[
\|D_\vp^\clm\|=\|\vp\|_\infty,
\]
characterize compactness, and show that the natural multiplication map by $h$ identifies the weighted and classical theories precisely when $h$ is inner. We also establish complex symmetry and obtain block matrix, defect, and semi-commutator identities via weighted truncated Hankel operators. As a main algebraic consequence, we prove
\[
D_\vp^\clm D_\psi^\clm=0
\quad\Longleftrightarrow\quad
\vp=0\ \text{or}\ \psi=0
\quad\text{a.e. on }\T.
\]
Finally, we derive a a rank-at-most-two correction formula and a finite-rank displacement identity for the generalized dual shift
$D_z^\clm$.
\end{abstract}
\maketitle

\section{Introduction}

Toeplitz operators on the Hardy space $H^2$ over the unit circle $\T$ form one of the central classes in function-theoretic operator theory. Their algebraic, spectral, and geometric properties connect complex function theory, harmonic analysis, and operator theory. Classical results such as
the Brown--Halmos theorem, Coburn's alternative, and the study of finite sums and products of Toeplitz operators have led to a large body of work; see, for example, \cite{Brown:Douglas, Ding, Gorkin:Zheng, Guo:Zheng, Halmos:PI, Zheng}. The interaction between Toeplitz operators, Hankel operators, compactness, essential spectra, and Fredholm theory continues to play a fundamental role in the modern development of the subject; see also \cite{Douglas:Book, Garcia:PIM, Ma:Yan:Zheng, Peller:Book}.

A major refinement of Toeplitz theory arises by replacing the Hardy space with a model space. If $u$ is a nonconstant inner function, the model space
\[
\clk_u=H^2\ominus uH^2
\]
is a closed backward-shift invariant subspace of $H^2$. The compression of a multiplication operator to $\clk_u$ is a truncated Toeplitz operator (TTO), introduced and studied systematically by Sarason \cite{Sarason2007}. The orthogonal complement of a model space gives rise to a dual theory. Since
\[
\clk_u^\perp=L^2(\T)\ominus\clk_u=uH^2\oplus H^2_{-},
\]
one can compress multiplication operators to $\clk_u^\perp$. For $\vp\in L^\infty(\T)$, the corresponding dual truncated Toeplitz operator (DTTO) is
\[
D_\vp=\restr{(I-P_u)M_\vp}{\clk_u^\perp}.
\]
Ding and Sang \cite{Sang:Ding} initiated a systematic study of these operators, obtaining block matrix representations, norm and compactness results, multiplicative identities, and zero-product phenomena. Further work on DTTOs and related compressions includes \cite{Bhuia:Nag, Bhuia:Ramesh:Nag, Camara:Ptak, Camara:Klis:Ptak:2020, Camara:Klis:Ptak:2022, Camara-Ross, Gu, HuEtAl2018, LiSangDing2020, Sang:Qin:Ding:2019, Sang:Qin:Ding, Wang:Zhao:Zheng, Zhu:Wu:Chen:2025}.

Parallel to this theory is the study of nearly $S^*$-invariant subspaces. A closed subspace $\clm\subset H^2$ is called nearly $S^*$-invariant if
\[
f\in\clm,\quad f(0)=0 \quad\Longrightarrow\quad S^*f\in\clm .
\]
By Hitt's theorem \cite{Hitt1988}, every nonzero proper nearly $S^*$-invariant subspace has the form
\[
\clm=h\clk_u,
\]
where $u$ is an inner function with $u(0)=0$, and $h$ is the extremal function. Multiplication by $h$ is an isometry from $\clk_u$ onto $\clm$, but in general it does not extend to a unitary operator on $L^2(\T)$. Sarason's description of isometric multipliers \cite{Sarason1987} and
Hayashi's work on Toeplitz kernels \cite{Hayashi1990} show that these spaces are closely related to Toeplitz kernels and exposed points of the unit ball of $H^1$. For background on model spaces, their operators, and complex symmetric operators, see \cite{GRM, Garcia:Putinar, NF}; for recent work on compressed shifts on nearly invariant subspaces, see \cite{LiangPartington2026}.

The purpose of this paper is to develop the basic operator theory of these compressions and to isolate the mechanism that appears when $\clk_u^\perp$ is replaced by
\[
\clm^\perp=L^2(\T)\ominus h\clk_u.
\]

For $\vp\in L^\infty(\T)$, we define the \emph{weighted dual truncated Toeplitz operator}
\[
D_\vp^\clm=\restr{P_{\clm^\perp}M_\vp}{\clm^\perp}.
\]
When $h\equiv1$, this recovers the classical DTTO $D_\vp$ on $\clk_u^\perp$. For a general extremal multiplier $h$, however, $D_\vp^\clm$ is not, in general, obtained from the classical DTTO through the natural multiplication map associated with $h$. The Hartmann--Ross projection formula
\[
P_\clm f=hP_u(\bar h f)
\]
provides the correct substitute for this missing unitary equivalence; more precisely, if
\[
V:\clk_u\to L^2(\T),\qquad Vk=hk,
\]
is regarded as an isometry with range $\clm$, then
\[
P_\clm=VV^*.
\]

We first establish the basic structure of $D_\vp^\clm$. We prove
\[
\|D_\vp^\clm\|=\|\vp\|_\infty
\]
and show that $D_\vp^\clm$ is compact only for the zero symbol. We then prove a rigidity theorem: multiplication by $h$ implements a unitary equivalence between $D_\vp^\clm$ and the classical DTTO $D_\vp$ precisely when $h$ is inner. We also show that $D_\vp^\clm$ is complex symmetric with
respect to a natural conjugation obtained by twisting the model-space conjugation by $h/\bar h$.

The second part of the paper develops the structural calculus. We obtain block matrix representations with respect to both decompositions
\[
L^2(\T)=\clm\oplus\clm^\perp,
\qquad
\clm^\perp=(H^2\ominus\clm)\oplus H^2_-.
\]
We introduce weighted truncated Hankel operators
\[
H_\vp^\clm
=
P_{\clm^\perp}M_\vp V:
\clk_u\to\clm^\perp
\]
and derive defect and semi-commutator identities. These identities allow us to compare weighted products with the corresponding classical DTTO products. As a principal consequence, we prove that the weighted family has no nontrivial two-factor zero products:
\[
D_\vp^\clm D_\psi^\clm=0
\quad\Longleftrightarrow\quad
\vp=0\ \text{or}\ \psi=0
\quad\text{a.e. on }\T.
\]
This extends the zero-product theorem of Ding and Sang from model-space complements to complements of nearly $S^*$-invariant subspaces.

We finally study the generalized dual shift
\[
D_z^\clm
=
\restr{P_{\clm^\perp}M_z}{\clm^\perp}.
\]
We obtain an explicit rank-at-most-two correction formula and a finite-rank displacement identity involving the compressed shift.

The paper is organized as follows. Section~2 recalls the necessary background on Hardy spaces, model spaces, nearly $S^*$-invariant subspaces, and the Hartmann--Ross projection formula. Section~3 introduces weighted DTTOs and proves boundedness, compactness, rigidity, complex symmetry, and block matrix representations. Section~4 develops weighted truncated Hankel operators, establishes defect and semi-commutator identities, and proves the zero-product theorem. Section~5 studies the generalized dual shift and the associated finite-rank displacement formulas.

\section{Preliminaries}

Let $\D$ denote the open unit disk in $\C$, and $\T = \partial \D$ the unit circle. The space of square-integrable functions on $\T$ with normalized Lebesgue measure $dm = \frac{1}{2\pi} d\theta$ is denoted $L^2(\T)$.

The classical Hardy space $H^2 \subset L^2(\T)$ consists of functions with vanishing negative Fourier coefficients. Its orthogonal complement is the co-analytic space $H_{-}^2 = L^2(\T) \ominus H^2 = \ol{z H^2}$. The orthogonal projections onto $H^2$ and $H_{-}^2$ are denoted $P_{+}$ and $P_{-}$, respectively. 

For $\vp \in L^\infty(\T)$, the multiplication operator $M_\vp : L^2(\T) \to L^2(\T)$ is defined by 
\[M_\vp f = \vp f, \quad f\in L^2(\T). \]
The Toeplitz operator $T_\vp : H^2 \to H^2$ and Hankel operator $H_\vp : H^2 \to H_{-}^2$ are defined by 
\begin{align}
    T_\vp f &= P_{+}(\vp f),\quad f\in H^2; \\
    H_\vp f &= P_{-}(\vp f),\quad f\in H^2.
\end{align}
Their adjoints are given by $T_\vp^* = T_{\bar{\vp}}$ and $H_\vp^* = \restr{P_{+} M_{\bar{\vp}}}{H_{-}^2}$, respectively. 

Consequently, the unilateral forward shift on $H^2$ is the analytic Toeplitz operator $S = T_z$, and its adjoint, the backward shift, is $S^* f = T_{\bar{z}} f = \frac{f(z) - f(0)}{z}$ for $f\in H^2$.

Restricting the multiplication operator $M_\vp$ to $H_{-}^2$ yields two distinct spatial mappings. The compression mapping $H_{-}^2$ to itself defines the \emph{dual Toeplitz operator} $S_\vp = P_{-} \restr{M_\vp}{H_{-}^2}$. Also,
\begin{equation}
    H_{\bar{\vp}}^* = P_{+} \restr{M_\vp}{H_{-}^2}.
\end{equation}

Furthermore, the dual Toeplitz operator is unitarily equivalent to the Toeplitz operator. Let $U : H_{-}^2 \to H^2$ be the canonical unitary operator $(Uf)(z) = \bar{z}f(\bar{z})$, which isometrically maps the orthonormal basis $\bar{z}^n$ to $z^{n-1}$ for $n \ge 1$. More precisely,
\begin{equation}
    U S_\vp U^* = T_{\tilde{\vp}},
\end{equation}
where $\tilde{\vp}(z) = \vp(\bar{z})$. For a comprehensive background on $H^\infty$ functions, Toeplitz operators, and shift-invariant subspaces, we refer the reader to Bercovici \cite{Bercovici}.

A function $u \in H^\infty$ is an inner function if $|u(z)| = 1$ almost everywhere on $\T$. For any nonconstant inner function $u$, the corresponding \emph{model space} is $\clk_u = H^2 \ominus uH^2$. By Beurling's Theorem, these are precisely the proper, non-trivial closed $S^*$-invariant subspaces of $H^2$. As a closed subspace of the Hardy space, $\clk_u$ is naturally a reproducing kernel Hilbert space (RKHS) on the open unit disk $\D$, evaluating at any point $\lambda \in \D$ via the reproducing kernel
\[ k_\lambda^u(z) = \frac{1-\ol{u(\lambda)}u(z)}{1-\bar{\lambda}z},\quad z\in \D. \]
Furthermore, the model space is equipped with a canonical conjugation (anti-linear isometric involution) $C_u : \clk_u \to \clk_u$, defined by
\begin{equation}\label{eq:conjuModelspace}
    C_u f(z) = u(z)\bar{z}\ol{f(z)}, \quad z\in \T.
\end{equation}
The orthogonal projection from $L^2(\T)$ onto $\clk_u$ is denoted $P_u$.

A closed subspace $\clm \subset H^2$ is \emph{nearly $S^*$-invariant} if for every $f \in \clm$ satisfying $f(0) = 0$, the backward shift satisfies $S^* f \in \clm$. While standard $S^*$-invariant subspaces coincide exactly with model spaces, nearly $S^*$-invariant subspaces admit a broader geometric structure characterized by Hitt \cite{Hitt1988}.

\begin{theorem}[Hitt's Theorem]\cite{Hitt1988}\label{thm:hitt}
Let $\clm\subsetneq H^2$ be a nonzero closed nearly $S^*$-invariant subspace. Then there exists an inner function $u$ satisfying $u(0)=0$, and a unique extremal function $h\in\clm$ such that
\[
\clm=h\clk_u.
\]
The function $h$ satisfies $\|h\|_2=1$, $h(0)>0$, and
\[
h\perp\{f\in\clm:f(0)=0\}.
\]
Furthermore, multiplication by $h$ defines an isometry from $\clk_u$ onto $\clm$, that is,
\[
\|hk\|_2=\|k\|_2,\qquad k\in\clk_u.
\]
\end{theorem}

The element $h \in \clm$ is geometrically defined as the \emph{extremal function} because it uniquely maximizes the point evaluation at the origin across the unit sphere of the subspace, solving the extremal problem:
\begin{equation}\label{eqn:extremal}
\sup \{ \Re(g(0)) : g \in \clm, \, \|g\|_2 = 1 \}.
\end{equation}

\begin{remark}
Since $h\in\clm=h\clk_u$, there exists $k\in\clk_u$ such that $h=hk$. As the boundary zero set of the nonzero function $h\in H^2$ has measure zero, it follows that $k=1$ a.e. Hence $1\in\clk_u$, which is equivalent to $u(0)=0$. Consequently, the reproducing kernel for $\clk_u$ at the origin simplifies to $k_0^u(z) = 1$. 
    
Furthermore, as demonstrated by Sarason \cite{Sarason1987}, the condition $u(0) = 0$ dictates the structural form of the extremal multiplier. Every isometric multiplier on $\clk_u$ has the form
\[
h=\frac{a}{1-ub},
\]
where $a,b\in H^\infty$ lie in the unit ball and satisfy
$|a|^2+|b|^2=1$ a.e. on $\T$. As a consequence, $h\clk_u$ is a (closed) nearly invariant subspace of $H^2$ with extremal function $h$ as in \eqref{eqn:extremal}. Hayashi \cite{Hayashi1990} characterized those nearly
$S^*$-invariant subspaces that arise as kernels of Toeplitz operators. We do not require the precise form of this characterization here.
\end{remark}

\subsection{Classical operators on the model space}

For a symbol $\vp \in L^\infty(\T)$, the \emph{truncated Toeplitz operator} (TTO) $A_\vp : \clk_u \to \clk_u$ is defined by 
\[ A_\vp f = P_u (\vp f), \quad f \in \clk_u. \]
Systematically studied by Sarason \cite{Sarason2007}, the adjoint is given by $A_\vp^* = A_{\bar{\vp}}$. When $\vp(z) = z$, $A_z$ is the compressed shift on the model space, denoted $S_u = \restr{P_u M_z}{\clk_u}$.

The \emph{big truncated Hankel operator} $B_\vp : \clk_u \to \clk_u^\perp$ is defined by
\[ B_\vp f = (I-P_u) (\vp f), \quad f \in \clk_u. \]
Its adjoint $B_\vp^* : \clk_u^\perp \to \clk_u$ is defined by
\[ B_\vp^* g = P_u (\bar{\vp} g), \quad g \in \clk_u^\perp. \]

With respect to the orthogonal direct sum $L^2(\T) = \clk_u \oplus \clk_u^\perp$, the multiplication operator $M_\vp$ admits the $2 \times 2$ block matrix representation:
\[ M_\vp = \begin{bmatrix} A_\vp & B_{\bar{\vp}}^* \\ B_\vp & D_\vp \end{bmatrix}, \]
where \[D_\vp f = (I-P_{u})(\vp f),\quad f\in \clk_u^\perp\] is the \emph{dual truncated Toeplitz operator} (DTTO) acting on $\clk_u^\perp$. Since $\clk_u^\perp=uH^2\oplus H_{-}^2$, the operator $D_\vp$ has the block representation
\[
D_\vp
=
\begin{bmatrix}
M_uT_\vp M_{\bar u} & M_uH_{u\bar\vp}^{*}\\
H_{\vp u}M_{\bar u} & S_\vp
\end{bmatrix}
\]
with respect to $uH^2\oplus H_{-}^2$. Equivalently, if
\[
\mathcal{U}:H^2\oplus H_{-}^2\to\clk_u^\perp,
\qquad \mathcal{U}(f\oplus g)=uf\oplus g,
\]
then
\[
\mathcal{U}^*D_\vp\mathcal{U}
=\begin{bmatrix}
T_\vp & H_{u\bar\vp}^{*}\\
H_{\vp u} & S_\vp
\end{bmatrix}
\]
with respect to $H^2\oplus H_{-}^2$. This block-matrix structure is modified in the nearly invariant setting by the presence of the extremal multiplier.

\subsection{The space $\clm^\perp$}

Since the nearly $S^*$-invariant subspace $\clm = h\clk_u$ is contained within the Hardy space $H^2$, its orthogonal complement in $L^2(\T)$ splits as
\[ \clm^\perp = \cln \oplus H_{-}^2, \]
where $H_{-}^2 = \ol{zH^2}$ is the co-analytic space, and $\cln = H^2 \ominus \clm$ is the analytic defect space. Every element $f \in \clm^\perp$ decomposes uniquely as $f = f_{+} + f_{-}$, where $f_{-} \in H^2_{-}$ and $f_{+} \in \cln$.

The projection formula below is due to Hartmann and Ross
\cite[Lemma~2.4]{Hartmann:Ross}. Since the extremal multiplier $h$ need not be bounded, the expression $P_u(\bar h f)$ has to be interpreted carefully, through the reproducing-kernel pairing, rather than as the action of a bounded operator $P_uM_{\bar h}$ on all of $L^2(\T)$. We recall the formula in a form adapted to our setting. Let
\[
V:\clk_u\to \clm,\qquad Vk=hk,
\]
be the canonical unitary induced by Hitt's theorem. Then the projection onto $\clm$ is simply $P_\clm=VV^*$. This gives the pointwise formula
\[
P_\clm f=hP_u(\bar h f),
\]
where the right-hand side is understood via the reproducing-kernel identity defining $V^*f$.

\begin{lemma}\label{lem:proj}
Let $\clm=h\clk_u$ be a nearly $S^*$-invariant subspace of $H^2$, where $h$ is the extremal function and multiplication by $h$ is an isometry from $\clk_u$ onto $\clm$. Let
\[
V:\clk_u\to \clm,\qquad Vk=hk ,
\]
be this unitary map. Then the orthogonal projection of $L^2(\T)$ onto $\clm$ is
\[
P_\clm=VV^*.
\]
Equivalently, for every $f\in L^2(\T)$,
\[
P_\clm f=h\,P_u(\bar h f),
\]
where $P_u(\bar h f)$ is understood through the reproducing-kernel pairing
\[
\big(P_u(\bar h f)\big)(\lambda)
=
\langle \bar h f,k_\lambda^u\rangle
=
\langle f,hk_\lambda^u\rangle_{L^2},
\qquad \lambda\in\D.
\]
Here
\[
\langle \bar h f,k_\lambda^u\rangle
:=
\int_\T \bar h f\,\overline{k_\lambda^u}\,dm
\]
is the natural $L^1$--$L^\infty$ pairing. Consequently,
\[
P_{\clm^\perp}f=f-h\,P_u(\bar h f), \qquad f\in L^2(\T).
\]
In formal notation, one may write
\[
P_\clm=M_hP_uM_{\bar h}, \qquad  P_{\clm^\perp}=I-M_hP_uM_{\bar h},
\]
provided this is interpreted in the above sense, and not as a composition of bounded multiplication operators on all of $L^2(\T)$ unless $h\in H^\infty$.
\end{lemma}

\begin{proof}
Define
\[
V:\clk_u\to L^2(\T),\qquad Vk=hk .
\]
By Hitt's theorem, multiplication by $h$ is an isometry from $\clk_u$ onto $\clm=h\clk_u$. Hence $V$ is a unitary operator from $\clk_u$ onto $\clm$, viewed as a closed subspace of $L^2(\T)$. Therefore, by the standard Hilbert space formula for the orthogonal projection onto the range of an isometry,
\[
P_\clm=VV^*.
\]

We now identify $V^*$. Fix $f\in L^2(\T)$. For every $k\in\clk_u$,
\[
\langle k,V^*f\rangle_{\clk_u}
=
\langle Vk,f\rangle_{L^2}
=
\langle hk,f\rangle_{L^2}.
\]
Thus
\[
\left|\langle hk,f\rangle\right| \leq \|hk\|_2\|f\|_2 = \|k\|_2\|f\|_2.
\]
Hence the map
\[
k\mapsto \langle hk,f\rangle
\]
is a bounded linear functional on $\clk_u$. Therefore there exists a unique element of $\clk_u$, namely $V^*f$, such that
\[
\langle k,V^*f\rangle
=
\langle hk,f\rangle
\qquad (k\in\clk_u).
\]

For $\lambda\in\D$, take $k=k_\lambda^u$. Then
\[
(V^*f)(\lambda)
= \langle V^*f,k_\lambda^u\rangle = \langle f,hk_\lambda^u\rangle.
\]
Since $h\in H^2$ and $f\in L^2(\T)$, we have $\bar h f\in L^1(\T)$. Hence the last expression may be written as
\[
\left(P_u(\bar h f)\right)(\lambda)
:= \int_\T \bar h f\,\overline{k_\lambda^u}\,dm = \langle f,hk_\lambda^u\rangle_{L^2}.
\]
Thus, in the reproducing-kernel sense,
\[
V^*f=P_u(\bar h f).
\]
Consequently,
\[
P_\clm f = VV^*f = h\,P_u(\bar h f),
\qquad f\in L^2(\T).
\]
Finally,
\[
P_{\clm^\perp}=I-P_\clm,
\]
and therefore,
\[
P_{\clm^\perp}f
=
f-h\,P_u(\bar h f),
\qquad f\in L^2(\T).
\]
This completes the proof.
\end{proof}

Throughout the paper, we regard
\[
V:\clk_u\to L^2(\T),\qquad Vk=hk,
\]
as an isometry with
\[
\operatorname{ran}V=\clm.
\]
Equivalently, $V$ is a unitary from $\clk_u$ onto $\clm$. Its adjoint is
\[
V^*:L^2(\T)\to\clk_u,
\]
and
\[
V^*V=I_{\clk_u}, \qquad VV^*=P_\clm.
\]

By extending co-analytic functions in $H_{-}^2$ via their conjugate-analytic extensions into $\D$, the orthogonal complement $\clm^\perp = \cln \oplus H_{-}^2$ forms a Reproducing Kernel Hilbert Space. The projection onto the analytic defect space $\cln$ is $P_\cln = P_{+} - P_\clm$.

\begin{lemma}\label{lem:RKHS}
Let $\clm=h\clk_u$ be a nearly $S^*$-invariant subspace and put $\cln=H^2\ominus\clm$. Then $\clm^\perp=\cln\oplus H_{-}^2$.
If
\[
f_{-}(\zeta)
= \sum_{n\geq1}a_{-n}\bar\zeta^{\,n},\qquad \zeta\in\T,
\]
define its conjugate-analytic extension to $\D$ by
\[
f_{-}^\#(\lambda)
:=\sum_{n\geq1}a_{-n}\bar\lambda^{\,n},\qquad \lambda\in\D.
\]
For $f=f_{+}+f_{-}\in\clm^\perp$, where
$f_{+}\in\cln$ and $f_{-}\in H_{-}^2$, define
\[
f(\lambda)
:=f_{+}(\lambda)+f_{-}^\#(\lambda),\qquad \lambda\in\D.
\]
Under this identification, $\clm^\perp$ is a scalar-valued harmonic reproducing kernel Hilbert space on $\D$.

Its reproducing kernel is
\[
K_\lambda^{\clm^\perp}
=
R_\lambda^\cln+k_\lambda^-,
\]
where
\[
R_\lambda^\cln(z)
=
\frac{1}{1-\bar\lambda z}
-
\overline{h(\lambda)}\,h(z)
\frac{1-\overline{u(\lambda)}u(z)}{1-\bar\lambda z},
\]
and
\[
k_\lambda^-(z)
= \sum_{n=1}^{\infty}\lambda^n\bar z^{\,n}
= \frac{\lambda\bar z}{1-\lambda\bar z}.
\]
Thus
\[
K_\lambda^{\clm^\perp}(z)
= \left(\frac{1}{1-\bar\lambda z}-\overline{h(\lambda)}\,h(z)
\frac{1-\overline{u(\lambda)}u(z)}{1-\bar\lambda z}\right)
+\frac{\lambda\bar z}{1-\lambda\bar z}.
\]
\end{lemma}

\begin{proof}
Since $\clm\subset H^2$, we have
\[
\clm^\perp
=(H^2\ominus\clm)\oplus H_{-}^2
=\cln\oplus H_{-}^2.
\]

For
\[
f_{-}(\zeta)=\sum_{n\geq1}a_{-n}\bar\zeta^{\,n}\in H_{-}^2,
\]
the series
\[
f_{-}^\#(\lambda)
=
\sum_{n\geq1}a_{-n}\bar\lambda^{\,n}
\]
converges absolutely for every $\lambda\in\D$. Indeed,
\[
|f_{-}^\#(\lambda)|
\leq
\left(\sum_{n\geq1}|a_{-n}|^2\right)^{1/2}
\left(\sum_{n\geq1}|\lambda|^{2n}\right)^{1/2}
=
\|f_{-}\|_2
\frac{|\lambda|}{\sqrt{1-|\lambda|^2}}.
\]
Hence the evaluation functionals are bounded. Moreover, the resulting
harmonic realization is injective by uniqueness of the analytic and
conjugate-analytic expansions.

The kernel for the analytic defect space $\cln$ is obtained by projecting
the Szeg\H{o} kernel $k_\lambda$ onto $\cln$. Since
\[
P_\cln=P_+-P_\clm,
\]
we have
\[
R_\lambda^\cln
=
P_\cln k_\lambda
=
k_\lambda-P_\clm k_\lambda.
\]
Using the projection formula
\[
P_\clm f=hP_u(\bar h f),
\]
where the expression is understood in the reproducing-kernel sense, we
obtain
\[
P_\clm k_\lambda
=
hP_u(\bar h k_\lambda).
\]

For every $g\in\clk_u$,
\[
\langle g,\bar h k_\lambda\rangle
=
\langle hg,k_\lambda\rangle
=
h(\lambda)g(\lambda).
\]
On the other hand,
\[
\left\langle
g,\overline{h(\lambda)}k_\lambda^u
\right\rangle
=
h(\lambda)\langle g,k_\lambda^u\rangle
=
h(\lambda)g(\lambda).
\]
Consequently,
\[
P_u(\bar h k_\lambda)
=
\overline{h(\lambda)}k_\lambda^u.
\]
Therefore
\[
R_\lambda^\cln(z)
=
k_\lambda(z)
-
\overline{h(\lambda)}h(z)k_\lambda^u(z),
\]
and hence
\[
R_\lambda^\cln(z)
=
\frac{1}{1-\bar\lambda z}
-
\overline{h(\lambda)}h(z)
\frac{1-\overline{u(\lambda)}u(z)}
     {1-\bar\lambda z}.
\]

Next, consider the co-analytic part. Since the inner product is linear
in the first variable, define
\[
k_\lambda^-(z)
=
\sum_{n=1}^{\infty}\lambda^n\bar z^{\,n}
=
\frac{\lambda\bar z}{1-\lambda\bar z}.
\]
Then
\[
\begin{aligned}
\langle f_{-},k_\lambda^-\rangle
&=
\left\langle
\sum_{n\geq1}a_{-n}\bar z^{\,n},
\sum_{n\geq1}\lambda^n\bar z^{\,n}
\right\rangle  \\
&=
\sum_{n\geq1}a_{-n}\bar\lambda^{\,n}
=
f_{-}^\#(\lambda).
\end{aligned}
\]

Finally, if
\[
f=f_{+}+f_{-}\in\clm^\perp,
\qquad
f_{+}\in\cln,\quad f_{-}\in H_{-}^2,
\]
then
\[
\begin{aligned}
\left\langle f,K_\lambda^{\clm^\perp}\right\rangle
&=
\langle f_{+},R_\lambda^\cln\rangle
+
\langle f_{-},k_\lambda^-\rangle \\
&=
f_{+}(\lambda)+f_{-}^\#(\lambda) \\
&=
f(\lambda).
\end{aligned}
\]
Hence
\[
K_\lambda^{\clm^\perp}
=
R_\lambda^\cln+k_\lambda^-
\]
is the reproducing kernel.
\end{proof}

\section{Weighted dual truncated Toeplitz operators}

Replacing the model-space complement $\clk_u^\perp$ by the orthogonal complement $\clm^\perp$ of a nearly $S^*$-invariant subspace introduces the extremal multiplier $h$ into the compression structure.

\begin{definition}\label{def:weighted_DTTO}
Let $\clm=h\clk_u$ be a nearly $S^*$-invariant subspace of $H^2$. For $\vp\in L^\infty(\T)$, the \emph{weighted dual truncated Toeplitz operator} with symbol $\vp$ is the compression of $M_\vp$ to $\clm^\perp$:
\[
D_\vp^{\clm}=\restr{P_{\clm^\perp}M_\vp}{\clm^\perp}.
\]
Equivalently,
\[
D_\vp^{\clm}f
=
P_{\clm^\perp}(\vp f),
\qquad f\in\clm^\perp .
\]
\end{definition}

\begin{remark}
Let
\[
V:\clk_u\to\clm,\qquad Vk=hk,
\]
be the canonical unitary induced by Hitt's theorem. Then
\[
P_\clm=VV^*,
\qquad
P_{\clm^\perp}=I-VV^*.
\]
Hence, for $f\in\clm^\perp$,
\[
D_\vp^{\clm}f
=
\vp f-VV^*(\vp f).
\]
Using the projection formula of Hartmann--Ross, this may be written as
\[
D_\vp^{\clm}f
=
\vp f-hP_u(\bar h\vp f).
\]
Here $\bar h\vp f\in L^1(\T)$, and $P_u(\bar h\vp f)$ is understood through the reproducing-kernel pairing
\[
\big(P_u(\bar h\vp f)\big)(\lambda)
=
\langle \bar h\vp f,k_\lambda^u\rangle
=
\langle \vp f,hk_\lambda^u\rangle,
\qquad \lambda\in\D,
\]
where the pairing is understood in the sense of
Lemma~\ref{lem:proj}. Thus the notation $hP_u(\bar h\vp f)$ represents the concrete realization of the operator $VV^*(\vp f)$.
\end{remark}

\begin{remark}
If, in addition, $h\in H^\infty$, then
$\bar h\vp f\in L^2(\T)$ for every $f\in\clm^\perp$. Using
\[
P_u=P_+-M_uP_+M_{\bar u},
\]
we obtain
\[
D_\vp^\clm f
= \vp f - hP_+(\bar h\vp f) + huP_+(\bar u\,\bar h\vp f).
\]
Since $h\in H^\infty$, all three terms belong to $L^2(\T)$.

For a general extremal multiplier $h\in H^2$, this decomposition is not
used, because the two terms involving $P_+$ need not separately belong to $L^2(\T)$. The universally valid formula is
\[
D_\vp^\clm f=\vp f-VV^*(\vp f).
\]
\end{remark}

Finally, for every $\vp\in L^\infty(\T)$,
\[
(D_\vp^{\clm})^*=D_{\bar\vp}^{\clm}.
\]
Indeed, for $f,g\in\clm^\perp$,
\[
\langle D_\vp^{\clm}f,g\rangle
=
\langle P_{\clm^\perp}(\vp f),g\rangle
=
\langle \vp f,g\rangle
=
\langle f,\bar\vp g\rangle
=
\langle f,P_{\clm^\perp}(\bar\vp g)\rangle
=
\langle f,D_{\bar\vp}^{\clm}g\rangle .
\]

\subsection{Boundedness and compactness}

Since $\clm\subset H^2$, we have
\[
\clm^\perp=(H^2\ominus \clm)\oplus H^2_{-}.
\]
In particular, $H_{-}^2\subset \clm^\perp$. This allows us to recover the classical dual Toeplitz compression as a corner of $D_\vp^\clm$.

\begin{theorem}\label{thm:boundedness}
For $\vp\in L^\infty(\T)$, the weighted dual truncated Toeplitz operator $D_\vp^\clm$ is bounded on $\clm^\perp$, and
\[
\|D_\vp^\clm\|=\|\vp\|_\infty .
\]
\end{theorem}

\begin{proof}
Since $P_{\clm^\perp}$ is an orthogonal projection, for every
$f\in\clm^\perp$ we have
\[
\|D_\vp^\clm f\|
=
\|P_{\clm^\perp}(\vp f)\|
\leq
\|\vp f\|
\leq
\|\vp\|_\infty\|f\|.
\]
Hence
\[
\|D_\vp^\clm\|\leq \|\vp\|_\infty.
\]

For the reverse inequality, note that $H_{-}^2\subset\clm^\perp$. Hence
\[
P_{\clm^\perp}P_{-}=P_{-},
\qquad
P_{-}P_{\clm^\perp}=P_{-}.
\]
Therefore, as an operator on $H_{-}^2$,
\[
P_{-}D_\vp^\clm P_{-}
=
P_{-}P_{\clm^\perp}M_\vp P_{-}
=
P_{-}M_\vp P_{-}.
\]
The operator $P_{-}M_\vp P_{-}$ on $H_{-}^2$ is the classical dual Toeplitz operator $S_\vp$. It is unitarily equivalent to a Hardy-space Toeplitz operator $T_{\widetilde{\vp}}$, where
\[
\widetilde{\vp}(z)=\vp(\bar{z}).
\]
Consequently,
\[
\|P_{-}D_\vp^\clm P_{-}\|
=
\|S_\vp\|
=
\|T_{\widetilde{\vp}}\|
=
\|\widetilde{\vp}\|_\infty
=
\|\vp\|_\infty.
\]
Since
\[
\|P_{-}D_\vp^\clm P_{-}\|\leq \|D_\vp^\clm\|,
\]
we obtain
\[
\|\vp\|_\infty\leq \|D_\vp^\clm\|.
\]
Combining the two inequalities gives
\[
\|D_\vp^\clm\|=\|\vp\|_\infty.
\]
\end{proof}

\begin{theorem}\label{thm:compactness}
For $\vp\in L^\infty(\T)$, the operator $D_\vp^\clm$ is compact on
$\clm^\perp$ if and only if $\vp=0$ almost everywhere on $\T$.
\end{theorem}

\begin{proof}
If $\vp=0$ almost everywhere, then $D_\vp^\clm=0$, and hence
$D_\vp^\clm$ is compact.

Conversely, suppose that $D_\vp^\clm$ is compact on $\clm^\perp$. Since $H_{-}^2\subset\clm^\perp$, the compression of $D_\vp^\clm$ to $H_{-}^2$ is compact. But, as above,
\[
P_{-}D_\vp^\clm P_{-}
=
P_{-}M_\vp P_{-}.
\]
Thus the classical dual Toeplitz operator
\[
S_\vp=\restr{P_{-}M_\vp}{H_{-}^2}
\]
is compact on $H_{-}^2$.

The dual Toeplitz operator $S_\vp$ is unitarily equivalent to the Hardy
Toeplitz operator $T_{\widetilde{\vp}}$, where
\[
\widetilde{\vp}(z)=\vp(\bar{z}).
\]
Therefore $T_{\widetilde{\vp}}$ is compact on $H^2$. Since a Toeplitz operator on $H^2$ with bounded symbol is compact if and only if its symbol is zero almost everywhere, we get
\[
\widetilde{\vp}=0 \quad \text{a.e. on } \T.
\]
Hence
\[
\vp=0 \quad \text{a.e. on } \T.
\]
This completes the proof.
\end{proof}

\subsection{Unitary equivalence and rigidity}
\begin{theorem}\label{thm:rigidity}
Let $\clm=h\clk_u$ be a nonzero proper nearly $S^*$-invariant subspace
of $H^2$, where $h$ is the extremal function. For
$\vp\in L^\infty(\T)$, let
\[
D_\vp
=
\restr{(I-P_u)M_\vp}{\clk_u^\perp}
\]
and
\[
D_\vp^\clm
=
\restr{P_{\clm^\perp}M_\vp}{\clm^\perp}.
\]
Then the following statements are equivalent:
\begin{enumerate}
\item The pointwise multiplication map $f\mapsto hf$ is well defined on
$\clk_u^\perp$ and defines a unitary operator
\[
W_h:\clk_u^\perp\longrightarrow\clm^\perp,
\qquad
W_hf=hf.
\]
\item The function $h$ is inner.
\end{enumerate}
When these conditions hold,
\[
W_h=\restr{M_h}{\clk_u^\perp}
\]
and
\[
D_\vp^\clm W_h=W_hD_\vp,
\qquad
\vp\in L^\infty(\T).
\]
Thus $W_h$ implements a unitary equivalence between the weighted and
classical families of dual truncated Toeplitz operators.
\end{theorem}

\begin{proof}
Suppose first that $h$ is inner. Then $|h|=1$ almost everywhere on
$\T$, so the multiplication operator
\[
M_h:L^2(\T)\longrightarrow L^2(\T)
\]
is unitary. Since
\[
M_h\clk_u=h\clk_u=\clm,
\]
unitarity gives
\[
M_h(\clk_u^\perp)=\clm^\perp.
\]
Consequently,
\[
W_h:=\restr{M_h}{\clk_u^\perp}:
\clk_u^\perp\longrightarrow\clm^\perp
\]
is unitary.

We next prove the intertwining relation. Since $M_h$ is unitary and
$M_h\clk_u=\clm$, we have
\[
P_\clm=M_hP_uM_{\bar h}.
\]
Let $f\in\clk_u^\perp$. Then
\begin{align*}
D_\vp^\clm W_hf
&=
P_{\clm^\perp}(\vp hf)\\
&=
(I-P_\clm)(\vp hf)\\
&=
\vp hf-M_hP_uM_{\bar h}(\vp hf)\\
&=
\vp hf-hP_u(|h|^2\vp f)\\
&=
\vp hf-hP_u(\vp f)\\
&=
h\bigl(\vp f-P_u(\vp f)\bigr)\\
&=
W_hD_\vp f.
\end{align*}
Therefore,
\[
D_\vp^\clm W_h=W_hD_\vp
\qquad
(\vp\in L^\infty(\T)).
\]

Conversely, suppose that the pointwise multiplication map
\[
W_h:\clk_u^\perp\longrightarrow\clm^\perp,
\qquad
W_hf=hf,
\]
is unitary. Then
\[
\|hf\|_2=\|f\|_2,
\qquad
f\in\clk_u^\perp.
\]
Hence
\[
\int_\T (|h|^2-1)|f|^2\,dm=0,
\qquad
f\in\clk_u^\perp.
\]

Since
\[
H_{-}^2\subseteq\clk_u^\perp,
\]
the preceding identity holds for every $f\in H_{-}^2$. Let $p$ be an analytic polynomial and take
\[
f(z)=\bar z\,p(\bar z).
\]
Then $f\in H_{-}^2$ and, since $|z|=1$ on $\T$,
\[
|f(z)|^2=|p(\bar z)|^2.
\]
Therefore,
\[
\int_\T (|h|^2-1)|p(\bar z)|^2\,dm=0
\]
for every analytic polynomial $p$.

Applying the polarization identity to the sesquilinear form
\[
\mathfrak q(p,q)
:=
\int_\T
(|h|^2-1)p(\bar z)\overline{q(\bar z)}\,dm,
\]
we obtain
\[
\int_\T
(|h|^2-1)p(\bar z)\overline{q(\bar z)}\,dm
=
0
\]
for all analytic polynomials $p$ and $q$. The linear span of the
functions
\[
p(\bar z)\overline{q(\bar z)}
\]
contains all trigonometric polynomials. It follows that
\[
\int_\T (|h|^2-1)\tau\,dm=0
\]
for every trigonometric polynomial $\tau$.

Since trigonometric polynomials are uniformly dense in $C(\T)$ and
\[
|h|^2-1\in L^1(\T),
\]
we conclude that
\[
\int_\T (|h|^2-1)\tau\,dm=0,
\qquad
\tau\in C(\T).
\]
Therefore,
\[
|h|^2-1=0
\quad\text{almost everywhere on }\T.
\]
Thus
\[
|h|=1
\quad\text{almost everywhere on }\T.
\]
Since $h\in H^2$ has unimodular boundary values, $h$ is inner.
\end{proof}

\begin{remark}
Theorem~\ref{thm:rigidity} shows that the natural multiplication map by $h$ identifies the weighted theory with the classical DTTO theory precisely in the inner case. Indeed, if $h$ is inner, then multiplication by $h$ is unitary on $L^2(\T)$ and carries $\clk_u^\perp$ onto $\clm^\perp$. Hence $D_\vp^\clm$ is unitarily equivalent to the usual dual truncated Toeplitz operator $D_\vp$ on $\clk_u^\perp$.

For a general nearly $S^*$-invariant subspace $\clm=h\clk_u$, however, the extremal function $h$ need not be inner. Although multiplication by $h$ is an isometry from $\clk_u$ onto $\clm$ by Hitt's theorem, it does not in general extend to a unitary operator on $L^2(\T)$, nor does it give a unitary map from $\clk_u^\perp$ onto $\clm^\perp$. Thus, outside the inner case, the compression
\[
D_\vp^\clm=\restr{P_{\clm^\perp}M_\vp}{\clm^\perp}
\]
is not, in general, identified with the classical DTTO by multiplication by $h$. The extremal multiplier enters through the projection formula
\[
P_\clm f=hP_u(\bar h f),
\]
and produces a genuinely weighted version of the Ding--Sang dual
truncated Toeplitz operators.
\end{remark}

\subsection{Complex symmetry}

We next show that the weighted dual truncated Toeplitz operators are complex symmetric with respect to a natural conjugation obtained by twisting the model-space conjugation by the extremal function.

\begin{definition}
Let $\clm=h\clk_u$. Since $h\in H^2$ is nonzero, its boundary zero set has measure zero. Define
\[
\omega_\clm(z)
= u(z)\bar z\,\frac{h(z)}{\overline{h(z)}},
\qquad z\in\T,
\]
where the quotient is defined arbitrarily on the null set on which $h=0$. The operator
\[
J_\clm:L^2(\T)\to L^2(\T)
\]
is defined by
\[
J_\clm f
=
\omega_\clm\,\overline{f}
=
u\bar{z}\,\frac{h}{\bar h}\,\overline{f},
\qquad f\in L^2(\T).
\]
Let
\[
\widetilde C_u:L^2(\T)\to L^2(\T),
\qquad
\widetilde C_u f=u\bar z\,\bar f.
\]
Then $\widetilde C_u$ is a conjugation on $L^2(\T)$ whose restriction to $\clk_u$ is the canonical model-space conjugation $C_u$. Thus
\[
J_\clm f
=
\frac{h}{\bar h}\,\widetilde C_u f.
\]
\end{definition}

\begin{remark}
The quotient $h/\bar h$ is unimodular almost everywhere. Hence
$\omega_\clm$ is unimodular almost everywhere, and therefore $J_\clm$ is an anti-linear isometry on $L^2(\T)$.
\end{remark}

\begin{lemma}
The operator $J_\clm$ is a conjugation on $L^2(\T)$. Moreover,
$J_\clm$ leaves both $\clm$ and $\clm^\perp$ invariant.
\end{lemma}

\begin{proof}
The map $J_\clm$ is conjugate-linear. Since $|\omega_\clm|=1$ almost everywhere, it is isometric. Moreover,
\[
J_\clm^2 f
=
J_\clm(\omega_\clm\bar f)
=
\omega_\clm\,\overline{\omega_\clm\bar f}
=
|\omega_\clm|^2 f
=
f.
\]
Thus $J_\clm$ is a conjugation on $L^2(\T)$.

We next show that $J_\clm$ preserves $\clm$. Let $x=hk\in\clm$, where $k\in\clk_u$. Then
\[
J_\clm(hk)
=
u\bar{z}\,\frac{h}{\bar h}\,\overline{hk}
=
u\bar{z}\,\frac{h}{\bar h}\,\bar h\,\bar k
=
h(u\bar{z}\,\bar k)
=
hC_u k.
\]
Since $C_u k\in\clk_u$, it follows that
\[
J_\clm(hk)\in h\clk_u=\clm.
\]
Hence $J_\clm\clm\subseteq \clm$. Since $J_\clm^2=I$, we also have $J_\clm\clm=\clm$.

Now let $f\in\clm^\perp$ and $x\in\clm$. Using the standard identity for conjugations,
\[
\langle J_\clm f,x\rangle
=
\langle J_\clm x,f\rangle,
\]
and since $J_\clm x\in\clm$ while $f\perp\clm$, we obtain
\[
\langle J_\clm f,x\rangle=0.
\]
Thus $J_\clm f\in\clm^\perp$. Therefore, $J_\clm$ leaves
$\clm^\perp$ invariant as well.
\end{proof}

\begin{theorem}\label{thm:complex_symmetry}
For every $\vp\in L^\infty(\T)$, the weighted dual truncated Toeplitz
operator $D_\vp^\clm$ is complex symmetric with respect to the
conjugation $J_\clm|_{\clm^\perp}$; that is,
\[
J_\clm D_\vp^\clm J_\clm
=
(D_\vp^\clm)^*
\quad \text{on } \clm^\perp.
\]
\end{theorem}

\begin{proof}
Since $J_\clm$ leaves $\clm^\perp$ invariant, it commutes with the
orthogonal projection onto $\clm^\perp$ in the sense that
\[
J_\clm P_{\clm^\perp}
=
P_{\clm^\perp}J_\clm.
\]
Let $f\in\clm^\perp$. Then
\[
J_\clm D_\vp^\clm J_\clm f
=
J_\clm P_{\clm^\perp}M_\vp J_\clm f
=
P_{\clm^\perp}J_\clm(\vp J_\clm f).
\]
Since $J_\clm f=\omega_\clm\bar f$, we have
\[
J_\clm(\vp J_\clm f)
=
\omega_\clm\,\overline{\vp\,\omega_\clm\bar f}
=
\omega_\clm\,\bar\vp\,\bar\omega_\clm\,f
=
\bar\vp f.
\]
Therefore
\[
J_\clm D_\vp^\clm J_\clm f
=
P_{\clm^\perp}(\bar\vp f)
=
D_{\bar\vp}^\clm f.
\]

It remains only to identify the adjoint. For $f,g\in\clm^\perp$,
\[
\langle D_\vp^\clm f,g\rangle
=
\langle P_{\clm^\perp}(\vp f),g\rangle
=
\langle \vp f,g\rangle
=
\langle f,\bar\vp g\rangle
=
\langle f,P_{\clm^\perp}(\bar\vp g)\rangle
=
\langle f,D_{\bar\vp}^\clm g\rangle.
\]
Hence
\[
(D_\vp^\clm)^*=D_{\bar\vp}^\clm.
\]
Combining this with the previous identity gives
\[
J_\clm D_\vp^\clm J_\clm
=
(D_\vp^\clm)^*.
\]
Thus $D_\vp^\clm$ is complex symmetric with respect to
$J_\clm|_{\clm^\perp}$.
\end{proof}

\subsection{Block matrix representations}

We first decompose the ambient space as
\[
L^2(\T)=\clm\oplus\clm^\perp.
\]
As in the classical DTTO setting, the multiplication operator $M_\vp$
admits a $2\times 2$ block representation with respect to this
decomposition. The only new feature is that the top-left compression is
now governed by the extremal multiplier $h$.

Let
\[
V:\clk_u\to \clm,\qquad Vk=hk,
\]
be the canonical unitary induced by Hitt's theorem. For
$\vp\in L^\infty(\T)$, define the bounded operator
\[
A_\vp^\clm:\clk_u\to \clk_u
\]
by
\[
A_\vp^\clm:=V^*P_\clm M_\vp V.
\]
Equivalently,
\[
V A_\vp^\clm V^*
=
P_\clm M_\vp\big|_{\clm}.
\]

The operator $A_\vp^\clm$ is represented by the bounded sesquilinear
form
\[
\langle A_\vp^\clm k,l\rangle
=
\int_\T \vp |h|^2 k\bar l\,dm,
\qquad k,l\in\clk_u.
\]
Although this form resembles a truncated Toeplitz form with symbol $|h|^2\vp$, that function need only belong to $L^1(\T)$; we therefore retain the notation $A_\vp^\clm$.

\begin{definition}
Let $\clm=h\clk_u$. The \emph{weighted big truncated Hankel operator}
with symbol $\vp\in L^\infty(\T)$ is
\[
B_\vp^\clm:\clm\to\clm^\perp,\qquad
B_\vp^\clm=P_{\clm^\perp}M_\vp\big|_{\clm}.
\]
Equivalently,
\[
B_\vp^\clm=P_{\clm^\perp}M_\vp P_\clm\big|_{\clm}.
\]
Its adjoint is
\[
(B_\vp^\clm)^*
=
P_\clm M_{\bar\vp}\big|_{\clm^\perp}.
\]
\end{definition}

\begin{proposition}\label{prop:block_matrix}
With respect to the orthogonal decomposition
\[
L^2(\T)=\clm\oplus\clm^\perp,
\]
the multiplication operator $M_\vp$ has the block matrix representation
\[
M_\vp
=
\begin{bmatrix}
V A_\vp^\clm V^* & (B_{\bar\vp}^\clm)^*\\
B_\vp^\clm & D_\vp^\clm
\end{bmatrix}.
\]
Here
\[
D_\vp^\clm=P_{\clm^\perp}M_\vp\big|_{\clm^\perp}.
\]
\end{proposition}

\begin{proof}
The four blocks are the four compressions of $M_\vp$ with respect to
$L^2(\T)=\clm\oplus\clm^\perp$.

The top-left block is
\[
\restr{P_\clm M_\vp}{\clm}
=
V A_\vp^\clm V^*
\]
by the definition of $A_\vp^\clm$.

The bottom-left block is, by definition,
\[
\restr{P_{\clm^\perp}M_\vp}{\clm}
=
B_\vp^\clm.
\]

The top-right block is
\[
\restr{P_\clm M_\vp}{\clm^\perp}.
\]
Since
\[
(B_{\bar\vp}^\clm)^*
=
\restr{P_\clm M_\vp}{\clm^\perp},
\]
the top-right block is $(B_{\bar\vp}^\clm)^*$.

Finally, the bottom-right block is
\[
\restr{P_{\clm^\perp}M_\vp}{\clm^\perp}
=
D_\vp^\clm.
\]
This proves the block representation.
\end{proof}

\begin{theorem}\label{thm:block_relations}
For $\vp,\psi\in L^\infty(\T)$, the following identities hold:
\begin{enumerate}
\item
\[
V A_{\vp\psi}^\clm V^* - V A_\vp^\clm V^*\,V A_\psi^\clm V^* =(B_{\bar\vp}^\clm)^*B_\psi^\clm.
\]

Equivalently, on $\clk_u$,
\[
A_{\vp\psi}^\clm-A_\vp^\clm A_\psi^\clm = V^*(B_{\bar\vp}^\clm)^*B_\psi^\clm V.
\]

\item
\[
D_{\vp\psi}^\clm-D_\vp^\clm D_\psi^\clm = B_\vp^\clm (B_{\bar\psi}^\clm)^*.
\]

\item
\[
B_{\vp\psi}^\clm - B_\vp^\clm V A_\psi^\clm V^*
=D_\vp^\clm B_\psi^\clm.
\]

\item
\[
(B_{\overline{\vp\psi}}^\clm)^* - V A_\vp^\clm V^*(B_{\bar\psi}^\clm)^* = (B_{\bar\vp}^\clm)^*D_\psi^\clm.
\]
\end{enumerate}
\end{theorem}

\begin{proof}
We use the identity
\[
M_{\vp\psi}=M_\vp M_\psi.
\]
By Proposition~\ref{prop:block_matrix}, we have
\[
M_{\vp\psi}
=
\begin{bmatrix}
V A_{\vp\psi}^\clm V^* & (B_{\overline{\vp\psi}}^\clm)^*\\
B_{\vp\psi}^\clm & D_{\vp\psi}^\clm
\end{bmatrix},
\]
while
\[
M_\vp
=
\begin{bmatrix}
V A_\vp^\clm V^* & (B_{\bar\vp}^\clm)^*\\
B_\vp^\clm & D_\vp^\clm
\end{bmatrix}
\]
and
\[
M_\psi
=
\begin{bmatrix}
V A_\psi^\clm V^* & (B_{\bar\psi}^\clm)^*\\
B_\psi^\clm & D_\psi^\clm
\end{bmatrix}.
\]
Multiplying the two block matrices for $M_\vp$ and $M_\psi$ and comparing
the four entries with the block matrix for $M_{\vp\psi}$ gives the desired
identities.

Indeed, comparison of the top-left entry gives
\[
V A_{\vp\psi}^\clm V^*
=
V A_\vp^\clm V^*\,V A_\psi^\clm V^*
+
(B_{\bar\vp}^\clm)^*B_\psi^\clm,
\]
which proves (1).

Comparison of the bottom-right entry gives
\[
D_{\vp\psi}^\clm
=
B_\vp^\clm(B_{\bar\psi}^\clm)^*
+
D_\vp^\clm D_\psi^\clm,
\]
which proves (2).

Comparison of the bottom-left entry gives
\[
B_{\vp\psi}^\clm
=
B_\vp^\clm V A_\psi^\clm V^*
+
D_\vp^\clm B_\psi^\clm,
\]
which proves (3).

Finally, comparison of the top-right entry gives
\[
(B_{\overline{\vp\psi}}^\clm)^*
=
V A_\vp^\clm V^*(B_{\bar\psi}^\clm)^*
+
(B_{\bar\vp}^\clm)^*D_\psi^\clm,
\]
which proves (4).
\end{proof}

\begin{remark}
When $h$ is inner, multiplication by $h$ is unitary on $L^2(\T)$ and the
above block matrix is unitarily equivalent to the classical Ding--Sang
block matrix for DTTOs. For a general extremal multiplier $h$, however,
the top-left block is not simply an ordinary truncated Toeplitz operator
with a bounded symbol. It is the weighted compression $A_\vp^\clm$,
encoded by the sesquilinear form
\[
\langle A_\vp^\clm k,l\rangle
=
\int_\T \vp |h|^2 k\bar l\,dm.
\]
Thus the weighted Hankel operators $B_\vp^\clm$ retain genuinely new
information coming from the extremal function $h$.
\end{remark}

We now decompose $D_\vp^\clm$ with respect to the orthogonal direct sum
\[
\clm^\perp=\cln\oplus H^2_{-},
\qquad
\cln=H^2\ominus\clm.
\]

\begin{theorem}\label{thm:N_Hminus_block}
With respect to the decomposition
\[
\clm^\perp=\cln\oplus H^2_{-},
\]
the operator $D_\vp^\clm$ has the block matrix representation
\[
D_\vp^\clm
=
\begin{bmatrix}
T_\vp^\cln & P_\cln H_{\bar\vp}^*\\
H_\vp|_{\cln} & S_\vp
\end{bmatrix},
\]
where
\[
T_\vp^\cln=\restr{P_\cln M_\vp}{\cln},
\]
\[
S_\vp=\restr{P_{-}M_\vp}{H_{-}^2},
\]
\[
H_\vp=\restr{P_{-}M_\vp}{H^2},
\]
and
\[
H_{\bar\vp}^*=\restr{P_{+}M_\vp}{H_{-}^2}.
\]
\end{theorem}

\begin{proof}
Since $\clm\subset H^2$, we have
\[
\clm^\perp=(H^2\ominus\clm)\oplus H^2_{-}
=
\cln\oplus H^2_{-}.
\]
Hence
\[
P_{\clm^\perp}=P_\cln+P_{-}.
\]
Expanding the compression gives
\[
D_\vp^\clm
=
(P_\cln+P_{-})M_\vp(P_\cln+P_{-}).
\]
Thus
\[
D_\vp^\clm=P_\cln M_\vp P_\cln+P_\cln M_\vp P_{-} +
P_{-}M_\vp P_\cln + P_{-}M_\vp P_{-}.
\]
The four entries are as follows. The top-left entry is
\[
\restr{P_\cln M_\vp}{\cln}=T_\vp^\cln.
\]
The bottom-right entry is
\[
\restr{P_{-}M_\vp}{H_{-}^2}=S_\vp.
\]
Since $\cln\subset H^2$, the bottom-left entry is
\[
\restr{P_{-}M_\vp}{\cln}=\restr{H_\vp}{\cln}.
\]
Finally, since $P_\cln=P_\cln P_{+}$, the top-right entry is
\[
\restr{P_\cln M_\vp}{H_{-}^2} =\restr{P_\cln P_{+}M_\vp}{H_{-}^2}
=P_\cln H_{\bar\vp}^*.
\]
This proves the block representation.
\end{proof}

\section{Algebraic relations and zero products}

To express the defect relations in a form comparable with the classical DTTO identities, we pull back the weighted big truncated Hankel operator along the canonical unitary
\[
V:\clk_u\to \clm,\qquad Vk=hk.
\]
This gives an operator from the model space $\clk_u$ into the orthogonal complement $\clm^\perp$.

\begin{definition}\label{def:weighted_hankel}
For $\vp\in L^\infty(\T)$, the \emph{weighted truncated Hankel operator} associated with $\clm=h\clk_u$ is
\[
H_\vp^\clm:\clk_u\to \clm^\perp,
\qquad
H_\vp^\clm=B_\vp^\clm V
=
P_{\clm^\perp}M_\vp V.
\]
Its adjoint is
\[
(H_\vp^\clm)^*
=
V^*(B_\vp^\clm)^*
=
V^*M_{\bar\vp}P_{\clm^\perp}.
\]
\end{definition}

\begin{proposition}\label{prop:weighted-semicommutator}
For every $\vp,\psi\in L^\infty(\T)$,
\[
D_{\vp\psi}^\clm
-
D_\vp^\clm D_\psi^\clm
=
H_\vp^\clm(H_{\bar\psi}^\clm)^*.
\]
\end{proposition}

\begin{proof}
By Theorem~\ref{thm:block_relations}\textup{(2)},
\[
D_{\vp\psi}^\clm-D_\vp^\clm D_\psi^\clm
=
B_\vp^\clm(B_{\bar\psi}^\clm)^*.
\]
Since
\[
H_\vp^\clm=B_\vp^\clm V
\]
and
\[
(H_{\bar\psi}^\clm)^*
=
V^*(B_{\bar\psi}^\clm)^*,
\]
we have
\begin{align*}
H_\vp^\clm(H_{\bar\psi}^\clm)^*
&=
B_\vp^\clm VV^*(B_{\bar\psi}^\clm)^* \\
&=
B_\vp^\clm P_\clm(B_{\bar\psi}^\clm)^*.
\end{align*}
Since $(B_{\bar\psi}^\clm)^*$ has range in $\clm$,
\[
P_\clm(B_{\bar\psi}^\clm)^*
=
(B_{\bar\psi}^\clm)^*.
\]
Consequently,
\[
H_\vp^\clm(H_{\bar\psi}^\clm)^*
=
B_\vp^\clm(B_{\bar\psi}^\clm)^*,
\]
and the result follows.
\end{proof}
The following identities are the weighted analogues of the standard DTTO defect relations.

\begin{theorem}\label{thm:defect}
For $\vp\in L^\infty(\T)$, the weighted dual truncated Toeplitz operator
$D_\vp^\clm$ satisfies
\[
D_\vp^\clm(D_\vp^\clm)^*
=
D_{|\vp|^2}^\clm
-
H_\vp^\clm(H_\vp^\clm)^*
\]
and
\[
(D_\vp^\clm)^*D_\vp^\clm
=
D_{|\vp|^2}^\clm
-
H_{\bar\vp}^\clm(H_{\bar\vp}^\clm)^* .
\]
\end{theorem}

\begin{proof}
Since
\[
I=P_{\clm^\perp}+P_\clm,
\]
we have
\begin{align*}
D_{|\vp|^2}^\clm
&=
P_{\clm^\perp}M_{|\vp|^2}P_{\clm^\perp}  \\
&=
P_{\clm^\perp}M_\vp
(P_{\clm^\perp}+P_\clm)
M_{\bar\vp}P_{\clm^\perp} \\
&=
P_{\clm^\perp}M_\vp P_{\clm^\perp}M_{\bar\vp}P_{\clm^\perp}
+
P_{\clm^\perp}M_\vp P_\clm M_{\bar\vp}P_{\clm^\perp} \\
&=
D_\vp^\clm D_{\bar\vp}^\clm
+
P_{\clm^\perp}M_\vp P_\clm M_{\bar\vp}P_{\clm^\perp}.
\end{align*}
Since
\[
(D_\vp^\clm)^*=D_{\bar\vp}^\clm
\]
and
\[
P_\clm=VV^*,
\]
the second term becomes
\begin{align*}
P_{\clm^\perp}M_\vp P_\clm M_{\bar\vp}P_{\clm^\perp}
&=
P_{\clm^\perp}M_\vp VV^*M_{\bar\vp}P_{\clm^\perp} \\
&=
(P_{\clm^\perp}M_\vp V)(V^*M_{\bar\vp}P_{\clm^\perp}) \\
&=
H_\vp^\clm(H_\vp^\clm)^*.
\end{align*}
Therefore
\[
D_{|\vp|^2}^\clm
=
D_\vp^\clm(D_\vp^\clm)^*
+
H_\vp^\clm(H_\vp^\clm)^*,
\]
which proves the first identity.

For the second identity, similarly,
\begin{align*}
D_{|\vp|^2}^\clm
&=
P_{\clm^\perp}M_{\bar\vp}
(P_{\clm^\perp}+P_\clm)
M_\vp P_{\clm^\perp} \\
&=
P_{\clm^\perp}M_{\bar\vp}P_{\clm^\perp}M_\vp P_{\clm^\perp}
+
P_{\clm^\perp}M_{\bar\vp}P_\clm M_\vp P_{\clm^\perp} \\
&=
D_{\bar\vp}^\clm D_\vp^\clm
+
P_{\clm^\perp}M_{\bar\vp}VV^*M_\vp P_{\clm^\perp} \\
&=
(D_\vp^\clm)^*D_\vp^\clm
+
(P_{\clm^\perp}M_{\bar\vp}V)(V^*M_\vp P_{\clm^\perp}) \\
&=
(D_\vp^\clm)^*D_\vp^\clm
+
H_{\bar\vp}^\clm(H_{\bar\vp}^\clm)^*.
\end{align*}
Rearranging gives the second identity.
\end{proof}

\begin{corollary}
If $\vp$ is unimodular, that is, $|\vp|=1$ almost everywhere on $\T$, then
\[
D_{|\vp|^2}^\clm=I_{\clm^\perp}.
\]
Consequently,
\[
I_{\clm^\perp}
-
D_\vp^\clm(D_\vp^\clm)^*
=
H_\vp^\clm(H_\vp^\clm)^*
\]
and
\[
I_{\clm^\perp}
-
(D_\vp^\clm)^*D_\vp^\clm
=
H_{\bar\vp}^\clm(H_{\bar\vp}^\clm)^*.
\]
\end{corollary}

\subsection{Zero products}
\begin{lemma}\label{lem:hankel-shift-vanishing}
For every $\alpha\in L^\infty(\T)$,
\[
H_\alpha S^n\longrightarrow0
\]
strongly on $H^2$.
\end{lemma}

\begin{proof}
Let
\[
p(z)=\sum_{j=0}^{d}c_jz^j
\]
be an analytic polynomial. Then
\[
H_\alpha S^np
=
\sum_{j=0}^{d}c_jP_-(\alpha z^{n+j}).
\]
Moreover,
\[
\|P_-(\alpha z^{n+j})\|_2^2
=
\sum_{m<-n-j}|\widehat\alpha(m)|^2
\longrightarrow0
\]
as $n\to\infty$, because $\alpha\in L^\infty(\T)\subset L^2(\T)$.
Therefore
\[
\|H_\alpha S^np\|\longrightarrow0.
\]

Since analytic polynomials are dense in $H^2$ and
\[
\|H_\alpha S^n\|\leq\|\alpha\|_\infty
\]
for every $n$, the conclusion follows by approximation.
\end{proof}

\begin{theorem}[Zero-product theorem]
\label{thm:zero-product}
Let $\clm=h\clk_u$ be a nonzero proper nearly $S^*$-invariant
subspace, where $u$ is a nonconstant inner function with $u(0)=0$.
For $\vp,\psi\in L^\infty(\T)$,
\[
D_\vp^\clm D_\psi^\clm=0
\]
if and only if
\[
\vp=0
\qquad\text{or}\qquad
\psi=0
\]
almost everywhere on $\T$.
\end{theorem}

\begin{proof}
The converse implication is immediate. Suppose that
\[
D_\vp^\clm D_\psi^\clm=0.
\]
Put
\[
\cln=H^2\ominus\clm,
\qquad
\clm^\perp=\cln\oplus H_{-}^2.
\]
Since $H_{-}^2\subseteq\clm^\perp$, we have
\[
P_{\clm^\perp}P_-=P_-P_{\clm^\perp}=P_-.
\]
Consequently,
\begin{align*}
0
&=
P_-D_\vp^\clm D_\psi^\clm P_- \\
&=
P_-M_\vp P_{\clm^\perp}M_\psi P_- \\
&=
P_-M_\vp P_\cln M_\psi P_-
+
P_-M_\vp P_-M_\psi P_-.
\end{align*}
Thus
\begin{equation}\label{eq:zero-product-corner}
S_\vp S_\psi = - P_-M_\vp P_\cln M_\psi P_-,
\end{equation}
where
\[
S_\alpha=P_-M_\alpha\big|_{H_{-}^2}, \qquad \alpha\in L^\infty(\T).
\]

Let $f,g\in H_{-}^2$ and define
\[
f_n=\bar z^{\,n}f, \qquad g_n=\bar z^{\,n}g.
\]
Then \eqref{eq:zero-product-corner} gives
\[
\langle S_\vp S_\psi f_n,g_n\rangle
= - \langle P_-M_\vp P_\cln M_\psi f_n,g_n\rangle.
\]
Since $\cln\subseteq H^2$,
\[
\|P_\cln M_\psi f_n\| \leq \|P_+(\psi\bar z^{\,n}f)\|.
\]
Writing $w=\psi f\in L^2(\T)$, Parseval's identity yields
\[
\|P_+(w\bar z^{\,n})\|_2^2
= \sum_{k\geq0}|\widehat w(k+n)|^2
\longrightarrow0.
\]
Hence
\[
P_\cln M_\psi f_n\longrightarrow0
\]
in $L^2(\T)$, and therefore
\[
\lim_{n\to\infty}
\langle S_\vp S_\psi f_n,g_n\rangle = 0.
\]

Let
\[
U:H_{-}^2\to H^2, \qquad U(\bar z^k)=z^{k-1},
\quad k\geq1.
\]
Then
\[
US_\alpha U^*=T_{\widetilde\alpha},
\qquad \widetilde\alpha(z)=\alpha(\bar z),
\]
and
\[
Uf_n=z^nUf, \qquad Ug_n=z^nUg.
\]
It follows that
\[
\langle S_\vp S_\psi f_n,g_n\rangle
=
\left\langle
T_{\widetilde\vp}T_{\widetilde\psi}z^nUf,
z^nUg
\right\rangle.
\]
Using
\[
T_{\widetilde\vp}T_{\widetilde\psi}
=
T_{\widetilde\vp\widetilde\psi}
-
H_{\overline{\widetilde\vp}}^*
H_{\widetilde\psi}
\]
and
\[
S^{*n}T_\eta S^n=T_\eta,
\qquad \eta\in L^\infty(\T),
\]
we obtain
\begin{align*}
\left\langle
T_{\widetilde\vp}T_{\widetilde\psi}z^nUf,
z^nUg \right\rangle
&= \left\langle T_{\widetilde\vp\widetilde\psi}Uf,Ug \right\rangle - \left\langle H_{\overline{\widetilde\vp}}^*
H_{\widetilde\psi}z^nUf, z^nUg \right\rangle.
\end{align*}
By Lemma~\ref{lem:hankel-shift-vanishing}, the second term tends to zero. Hence
\[
\left\langle
T_{\widetilde\vp\widetilde\psi}Uf,
Ug
\right\rangle
=
0.
\]
Since $f,g\in H_{-}^2$ are arbitrary and $U$ is unitary,
\[
T_{\widetilde\vp\widetilde\psi}=0.
\]
A Toeplitz operator with bounded symbol is zero only when its symbol is
zero almost everywhere. Therefore
\[
\widetilde\vp\,\widetilde\psi=0
\quad\text{a.e.},
\]
and consequently
\begin{equation}\label{eq:symbol-product-zero}
\vp\psi=0
\quad\text{a.e. on }\T.
\end{equation}

By Proposition~\ref{prop:weighted-semicommutator},
\[
D_{\vp\psi}^\clm-D_\vp^\clm D_\psi^\clm
=
H_\vp^\clm(H_{\bar\psi}^\clm)^*.
\]
In view of the assumption and \eqref{eq:symbol-product-zero}, we obtain
\[
H_\vp^\clm(H_{\bar\psi}^\clm)^*=0.
\]

Let
\[
B_\alpha
=
P_{\clk_u^\perp}M_\alpha\big|_{\clk_u}
\]
be the classical big truncated Hankel operator. We claim that
\[
\ker H_\alpha^\clm=\ker B_\alpha.
\]
Indeed, for $k\in\clk_u$,
\begin{align*}
k\in\ker H_\alpha^\clm
&\iff
P_{\clm^\perp}(\alpha hk)=0 \\
&\iff
\alpha hk\in h\clk_u \\
&\iff
\alpha k\in\clk_u \\
&\iff
k\in\ker B_\alpha,
\end{align*}
where we have used that $h$ is nonzero almost everywhere.

It follows that
\[
\overline{\ran B_{\bar\psi}^*}
=
(\ker B_{\bar\psi})^\perp
=
(\ker H_{\bar\psi}^\clm)^\perp
=
\overline{\ran (H_{\bar\psi}^\clm)^*}.
\]
Since
\[
H_\vp^\clm(H_{\bar\psi}^\clm)^*=0,
\]
we therefore get
\[
\overline{\ran B_{\bar\psi}^*}
\subseteq
\ker B_\vp,
\]
and hence
\[
B_\vp B_{\bar\psi}^*=0.
\]

For the classical DTTOs,
\[
D_{\vp\psi}-D_\vp D_\psi
=
B_\vp B_{\bar\psi}^*.
\]
Together with \(\vp\psi=0\), this gives
\[
D_\vp D_\psi=0.
\]
The zero-product theorem of Ding and Sang
\cite[Theorem~3.4]{Sang:Ding} now implies that
\[
\vp=0
\qquad\text{or}\qquad
\psi=0
\]
almost everywhere on $\T$.
\end{proof}

\section{The generalized dual shift}
\begin{lemma}\label{lem:commutators}
Assume that inner products are linear in the first variable. The commutators of $P_\clm$ with the bilateral shifts satisfy
\[
[P_\clm,M_z]
=
h\otimes \bar{z} h
-
hu\otimes \bar{z} hu
\]
and
\[
[M_{\bar{z}},P_\clm]
=
\bar{z} h\otimes h
-
\bar{z} hu\otimes hu.
\]
Here
\[
(f\otimes g)x=\langle x,g\rangle f.
\]
In particular, these commutators have rank at most two.
\end{lemma}

\begin{proof}
It is enough to prove the first identity on the dense subspace
$L^\infty(\T)\subset L^2(\T)$, since both sides define bounded operators on $L^2(\T)$.

Let $f\in L^\infty(\T)$. Then $\bar h f\in L^2(\T)$, and the classical identity
\[
[P_u,M_z]=1\otimes \bar z-u\otimes \bar z u
\]
may be applied to $\bar h f$. Using the projection formula
\[
P_\clm f=hP_u(\bar h f),
\]
we get
\begin{align*}
[P_\clm,M_z]f
&=
P_\clm(zf)-zP_\clm f  \\
&=
hP_u(z\bar h f)-zhP_u(\bar h f) \\
&=
h\big(P_u(z\bar h f)-zP_u(\bar h f)\big) \\
&=
h\langle \bar h f,\bar z\rangle
-
hu\langle \bar h f,\bar z u\rangle  \\
&=
h\langle f,\bar z h\rangle
-
hu\langle f,\bar zhu\rangle.
\end{align*}
Thus
\[
[P_\clm,M_z]
=
h\otimes \bar zh-hu\otimes \bar zhu.
\]
Taking adjoints gives
\[
[M_{\bar z},P_\clm]
=
\bar zh\otimes h-\bar zhu\otimes hu.
\]
\end{proof}

Let
\[
S_{\clm^\perp}=P_{\clm^\perp}M_z\big|_{\clm^\perp}.
\]
Then
\[
S_{\clm^\perp}^*
=
P_{\clm^\perp}M_{\bar{z}}\big|_{\clm^\perp}.
\]
Observe that
\[
S_{\clm^\perp}=D_z^\clm.
\]
\begin{theorem}\label{thm:shift_defect}
For $\vp\in L^\infty(\T)$, the operator $D_\vp^\clm$ satisfies
\[
S_{\clm^\perp}^*D_\vp^\clm S_{\clm^\perp}
-
D_\vp^\clm
=
\Delta_1+\Delta_2,
\]
where
\[
\Delta_1
=
-
P_{\clm^\perp}(\bar{z} h)\otimes \bar{z}\bar\vp h
+
P_{\clm^\perp}(\bar{z} hu)\otimes \bar{z}\bar\vp hu,
\]
and
\[
\Delta_2
=
-
S_{\clm^\perp}^*P_{\clm^\perp}(\vp h)\otimes \bar{z} h
+
S_{\clm^\perp}^*P_{\clm^\perp}(\vp hu)\otimes \bar{z} hu.
\]
Consequently,
\[
\operatorname{rank}
\left(
S_{\clm^\perp}^*D_\vp^\clm S_{\clm^\perp} - D_\vp^\clm
\right) \leq 4.
\]
\end{theorem}

\begin{proof}
Let $f\in\clm^\perp$. Since $P_\clm f=0$, we have
\[
S_{\clm^\perp}f = P_{\clm^\perp}(zf)
= zf-P_\clm(zf) = zf-[P_\clm,M_z]f.
\]
By Lemma~\ref{lem:commutators},
\[
S_{\clm^\perp}f =
zf - \langle f,\bar{z} h\rangle h
+ \langle f,\bar{z} hu\rangle hu.
\]
Applying $D_\vp^\clm$ gives
\[
D_\vp^\clm S_{\clm^\perp}f
=
P_{\clm^\perp}(\vp zf)
-
\langle f,\bar{z} h\rangle P_{\clm^\perp}(\vp h)
+
\langle f,\bar{z} hu\rangle P_{\clm^\perp}(\vp hu).
\]

We first apply $S_{\clm^\perp}^*$ to the first term. Since
\[
S_{\clm^\perp}^*
=
P_{\clm^\perp}M_{\bar{z}}\big|_{\clm^\perp},
\]
we have
\begin{align*}
S_{\clm^\perp}^*P_{\clm^\perp}(\vp zf)
&=
P_{\clm^\perp}M_{\bar{z}}P_{\clm^\perp}(\vp zf) \\
&=
P_{\clm^\perp}M_{\bar{z}}(I-P_\clm)(\vp zf) \\
&=
P_{\clm^\perp}(\vp f)
-
P_{\clm^\perp}M_{\bar{z}}P_\clm(\vp zf).
\end{align*}
Using
\[
M_{\bar{z}}P_\clm
=
[M_{\bar{z}},P_\clm]+P_\clm M_{\bar{z}}
\]
and
\[
P_{\clm^\perp}P_\clm=0,
\]
we get
\[
P_{\clm^\perp}M_{\bar{z}}P_\clm(\vp zf)
=
P_{\clm^\perp}[M_{\bar{z}},P_\clm](\vp zf).
\]
Therefore, by Lemma~\ref{lem:commutators},
\begin{align*}
&-P_{\clm^\perp}[M_{\bar{z}},P_\clm](\vp zf) \\
&=
-
\langle \vp zf,h\rangle P_{\clm^\perp}(\bar{z} h)
+
\langle \vp zf,hu\rangle P_{\clm^\perp}(\bar{z} hu) \\
&=
-
\langle f,\bar{z}\bar\vp h\rangle P_{\clm^\perp}(\bar{z} h)
+
\langle f,\bar{z}\bar\vp hu\rangle P_{\clm^\perp}(\bar{z} hu).
\end{align*}
Thus
\[
S_{\clm^\perp}^*P_{\clm^\perp}(\vp zf) = D_\vp^\clm f+\Delta_1 f.
\]

Next, applying $S_{\clm^\perp}^*$ to the remaining two terms gives
\begin{align*}
& S_{\clm^\perp}^*
\left( -\langle f,\bar{z} h\rangle P_{\clm^\perp}(\vp h)
+ \langle f,\bar{z} hu\rangle P_{\clm^\perp}(\vp hu) \right) \\
&= -\langle f,\bar{z} h\rangle S_{\clm^\perp}^*P_{\clm^\perp}(\vp h)
+ \langle f,\bar{z} hu\rangle S_{\clm^\perp}^*P_{\clm^\perp}(\vp hu) \\
&= \Delta_2 f.
\end{align*}

Combining the two parts, we obtain
\[
S_{\clm^\perp}^*D_\vp^\clm S_{\clm^\perp}f
= D_\vp^\clm f+\Delta_1f+\Delta_2f.
\]
Hence
\[
S_{\clm^\perp}^*D_\vp^\clm S_{\clm^\perp}
- D_\vp^\clm = \Delta_1+\Delta_2.
\]
Since each $\Delta_i$ has rank at most two, the defect has rank at most four.
\end{proof}

\subsection{Properties of the generalized dual shift}

The classical dual shift
\[
D_z^u=(I-P_u)M_z\big|_{\clk_u^\perp}
\]
acts on the orthogonal complement of the model space. In the present setting we replace $\clk_u^\perp$ by
\[
\clm^\perp=(h\clk_u)^\perp,
\]
and the extremal function $h$ enters the operator through the projection formula for $P_\clm$.

We first record the explicit action of the generalized dual shift.

\begin{proposition}\label{prop:dual_shift_action}
Let $\clm=h\clk_u$ be a nearly $S^*$-invariant subspace. For $f\in\clm^\perp$, the generalized dual shift
\[
D_z^\clm=P_{\clm^\perp}M_z\big|_{\clm^\perp}
\]
is given by
\[
D_z^\clm f
=
zf-\langle f,\bar{z} h\rangle h
+
\langle f,\bar{z} hu\rangle hu.
\]
\end{proposition}

\begin{proof}
Let $f\in\clm^\perp$. Since $P_\clm f=0$, we have
\[
D_z^\clm f
=
P_{\clm^\perp}(zf)
=
zf-P_\clm(zf).
\]
Moreover,
\[
P_\clm(zf)
=
[P_\clm,M_z]f,
\]
because $M_zP_\clm f=0$. By Lemma~\ref{lem:commutators},
\[
[P_\clm,M_z]
=
h\otimes \bar{z} h
-
hu\otimes \bar{z} hu.
\]
Hence
\[
P_\clm(zf)
=
\langle f,\bar{z} h\rangle h
-
\langle f,\bar{z} hu\rangle hu.
\]
Therefore,
\[
D_z^\clm f
=
zf-\langle f,\bar{z} h\rangle h
+
\langle f,\bar{z} hu\rangle hu.
\]
\end{proof}
Viewing both sides as operators from $\clm^\perp$ into $L^2(\T)$,
we have
\[
D_z^\clm-\restr{M_z}{\clm^\perp}
=
-h\otimes\bar zh
+
hu\otimes\bar zhu.
\]
Consequently,
\[
\operatorname{rank}
\left( D_z^\clm-\restr{M_z}{\clm^\perp} \right) \leq2.
\]
\begin{remark}
When $h\equiv 1$, we have $\clm=\clk_u$ and
$\clm^\perp=\clk_u^\perp$. In this case the formula becomes
\[
D_z^u f
=
zf-\langle f,\bar{z}\rangle 1
+
\langle f,\bar{z} u\rangle u.
\]
Because $u(0)=0$, the function $\bar zu$ belongs to $H^2$. Moreover, for
every $g\in H^2$,
\[
\langle ug,\bar zu\rangle=\langle zg,1\rangle=0,
\]
and $\bar zu\perp H_{-}^2$. Thus $\bar zu\in\clk_u$, so
\[
\langle f,\bar zu\rangle=0,
\qquad f\in\clk_u^\perp.
\]
Consequently,
\[
D_z^u f=zf-\langle f,\bar z\rangle 1,
\qquad f\in\clk_u^\perp,
\]
which is the usual single rank-one correction formula for the classical dual shift. In the weighted setting, the two correction terms involving $h$ and $hu$ need not collapse, so the compression of $M_z$ involves at most two rank-one correction terms.
\end{remark}

\section*{Declarations}
\noindent\textbf{Funding.} Not applicable.\\
\noindent\textbf{Conflict of interest.} The author declares no conflicts of interest.\\
\noindent\textbf{Data availability.} This work is purely theoretical.

\end{document}